\pdfoutput=1 
\documentclass[11pt]{amsart}
\usepackage[utf8]{inputenc}
\usepackage{amssymb,amsmath,amsthm,enumerate,enumitem,colonequals,mlmodern,tikz-cd,microtype}
\usepackage[cal=euler,bfcal,bb=px,bfbb]{mathalpha}
\usepackage[top=3.75cm, bottom=3cm, left=3.5cm, right=3.5cm]{geometry}

\makeatletter
\@namedef{subjclassname@2020}{\textup{2020} Mathematics Subject Classification}
\makeatother

\usepackage{xcolor}
\colorlet{darkblue}{blue!55!black}
\colorlet{darkcyan}{cyan!50!black}
\colorlet{darkgreen}{green!60!black}

\usepackage{hyperref}
\hypersetup{
    colorlinks=true,
    linkcolor= darkblue,
    urlcolor= darkcyan,
    citecolor= darkgreen,
}

\def\eqref#1{\textcolor{darkblue}{(\ref{#1})}}

\PassOptionsToPackage{hyphens}{url}
\usepackage{hyperref}

\usepackage[nameinlink]{cleveref} 
\Crefformat{section}{#2\S#1#3} 
\Crefmultiformat{section}{#2\S\S#1#3}{ and~#2#1#3}{, #2#1#3}{, and~#2#1#3}
\crefname{hypothesis}{hypothesis}{hypotheses}
\Crefname{hypothesis}{Hypothesis}{Hypotheses}

\usepackage[pagewise]{lineno}
\let\oldequation\equation
\let\oldendequation\endequation
\renewenvironment{equation}{\linenomathNonumbers\oldequation}{\oldendequation\endlinenomath}
\expandafter\let\expandafter\oldequationstar\csname equation*\endcsname
\expandafter\let\expandafter\oldendequationstar\csname endequation*\endcsname
\renewenvironment{equation*}{\linenomathNonumbers\oldequationstar}{\oldendequationstar\endlinenomath}
\let\oldalign\align
\let\oldendalign\endalign
\renewenvironment{align}{\linenomathNonumbers\oldalign}{\oldendalign\endlinenomath}
\expandafter\let\expandafter\oldalignstar\csname align*\endcsname
\expandafter\let\expandafter\oldendalignstar\csname endalign*\endcsname
\renewenvironment{align*}{\linenomathNonumbers\oldalignstar}{\oldendalignstar\endlinenomath}


\newcounter{intro}
\newcounter{HypCounter}

\newtheorem{introthm}[intro]{Theorem}

\newtheorem{introcor}[intro]{Corollary}

\theoremstyle{plain}
\newtheorem{theorem}{Theorem}[section]
\newtheorem{lemma}[theorem]{Lemma}
\newtheorem{corollary}[theorem]{Corollary}

\theoremstyle{definition}

\newtheorem{definition}[theorem]{Definition}
\newtheorem{example}[theorem]{Example}

\newtheorem*{hypothesis*}{Hypothesis}
\newtheorem{notation}[theorem]{Notation}

\newtheorem{remark}[theorem]{Remark}

\newtheorem*{ack}{Acknowledgements}

\numberwithin{equation}{section}
\numberwithin{theorem}{section}

\title[Descent, generation, and noncommutative coherent algebras]{Preservation for generation \\ along the structure morphism of \\ coherent algebras over a scheme}

\author[A.~Bhaduri]{Anirban Bhaduri}
\address{A.~Bhaduri,
Department of Mathematics,
University of South Carolina, 
Columbia, SC 29208,
U.S.A.}
\email{abhaduri@email.sc.edu}

\author[S.~Dey]{Souvik Dey}
\address{S.~Dey,
Faculty of Mathematics and Physics,
Department of Algebra,
Charles University, 
Sokolovsk\'{a} 83, 186 75 Praha, 
Czech Republic}
\email{souvik.dey@matfyz.cuni.cz}

\author[P.~Lank]{Pat Lank}
\address{P.~Lank,
Department of Mathematics,
University of South Carolina, 
Columbia, SC 29208,
U.S.A.}
\email{plankmathematics@gmail.com}

\date{\today}



\begin{document}

\begin{abstract}
    This work demonstrates classical generation is preserved by the derived pushforward along the structure morphism of a noncommutative coherent algebra to its underlying scheme.
    Additionally, we establish that the Krull dimension of a variety over a field is a lower bound for the Rouquier dimension of the bounded derived category associated with a noncommutative coherent algebra on it. This is an extension of a classical result of Rouquier to the noncommutative context.
\end{abstract}

\maketitle
\setcounter{tocdepth}{1}

\section{Introduction}
\label{sec:intro}

In this note, we examine the behavior of generation along the derived pushforward of the structure morphism from a noncommutative coherent algebra to its underlying scheme. As a result, our investigation introduces new methods coming from noncommutative homological techniques for establishing upper bounds on the Rouquier dimension within triangulated categories relevant to the commutative world.

Consider a Noetherian scheme $X$. The \textit{Rouquier dimension} of $D^b_{\operatorname{coh}}(X)$, denoted by $\dim D^b_{\operatorname{coh}}(X)$, is the smallest integer $n$ such that there exists an object $G$, where the smallest subcategory generated by $G$ using finite direct sums, summands, shifts, and at most $n+1$ cones coincides with $D^b_{\operatorname{coh}}(X)$. Any such subcategory is denoted as $\langle G \rangle_n$, and objects $G$ satisfying these conditions are called \textit{strong generators}. More generally, an object $G$ is a \textit{classical generator} if the smallest triangulated subcategory containing $G$ that is closed under direct summands coincides with $D^b_{\operatorname{coh}}(X)$. For further background, see Section~\ref{sec:generation}. 

There has been recent advancements that illuminate sufficient conditions under which $D^b_{\operatorname{coh}}(X)$ possesses such objects: for strong generators, any Noetherian quasi-excellent separated scheme of finite Krull dimension \cite{Aoki:2021}; for classical generators, any Noetherian $J\textrm{-}2$ scheme \cite{Elagin/Lunts/Schnurer:2020} or Notherian schemes whose closed subschemes admit open regular locus locus \cite{Dey/Lank:2024b}. Albeit being aware of the existence of such objects, explicit descriptions are limited to a few instances: smooth quasi-projective schemes over a field \cite{Rouquier:2008}, Frobenius pushforwards on a compact generator for Noetherian schemes of prime characteristic \cite{BILMP:2023}, singular varieties that admit a resolution of singularities \cite{Lank:2024, Dey/Lank:2024}, and objects arising from module categories \cite{Dey/Lank/Takahashi:2023}. 

Let's now focus on understanding the Rouquier dimension of $D^b_{\operatorname{coh}}(X)$ in the context of a singular variety $X$. In \cite{Lank:2024, Dey/Lank:2024}, it was demonstrated that if $\pi \colon \widetilde{X} \to X$ is a resolution of singularities and $G$ is a strong generator for $D^b_{\operatorname{coh}}(\widetilde{X})$, then $\mathbb{R}\pi_\ast G$ is a strong generator for $D^b_{\operatorname{coh}}(X)$. This provides effective methods for explicitly describing strong generators although the associated generation time remains difficult to establish.

Recent attention has been drawn to the case where $X$ is a projective curve \cite{Burban/Drozd:2011, Burban/Drozd/Gavran:2017, Hanlon/Hicks:2023,Burban/Drozd/Gavran:2017a} with significant progress made in establishing upper bounds on the Rouquier dimension of its bounded derived category. In particular, \cite{Burban/Drozd/Gavran:2017} identified an essentially surjective functor $D^b_{\operatorname{coh}}(\mathbb{X}) \to D^b_{\operatorname{coh}}(X)$, where $\mathbb{X}$ is a noncommutative projective curve, yielding a \textit{categorical resolution} of $X$. In the context of this work, a \textit{noncommutative scheme} is defined as a pair $(X, \mathcal{A})$, where $X$ is a Noetherian scheme, and $\mathcal{A}$ is a coherent $\mathcal{A}$-algebra \cite{Yekutieli/Zhang:2006, Burban/Drozd/Gavran:2017a, DeDeyn/Lank/ManaliRahul:2024b}. Additional background information can be found in Section~\ref{sec:noncommutative_schemes_background}. There alternative approaches to `noncommutative algebraic geometry' for which we refer the reader to \cite{Orlov:2016, Kontsevich/Rosenberg:2000, Laudal:2003, Bellamy/Rogalski/Schedler/Stafford/Wemyss:2016}.

In a closely related context, \cite{Elagin/Lunts/Schnurer:2020} demonstrated that the bounded derived category of coherent $\mathcal{A}$-modules, denoted $D^b_{\operatorname{coh}}(\mathcal{A})$, admits a classical generator whenever $\mathcal{A}$ is a coherent $\mathcal{O}_X$-algebra over a Noetherian $J\textrm{-}2$ scheme (see Remark~\ref{rmk:j_conditions} for details on terminology). There has been recent attention towards the existence of strong generators in this noncommutative setting \cite{DeDeyn/Lank/ManaliRahul:2024b}. 

Our work is motivated by two key aspects: the preservation of generation along a derived pushforward of a proper surjective morphism and the geometry of a noncommutative scheme. This represents a blend of ideas from \cite{Lank:2024, Dey/Lank:2024, Elagin/Lunts/Schnurer:2020, Burban/Drozd/Gavran:2017a} and leads us to our first result: an investigation into the preservation of classical generators through the derived pushforward along the structure morphism from a noncommutative scheme to its underlying scheme. This brings attention to our first result.

\begin{introthm}\label{introthm:descent_for_coherent_algebras}(see Theorem~\ref{thm:descent_for_coherent_algebras})
    Let $X$ be a Noetherian $J\textrm{-}2$ scheme of finite Krull dimension. Suppose $\mathcal{A}$ is a coherent $\mathcal{O}_X$-algebra with full support and canonical map $\pi \colon \mathcal{O}_X \to \mathcal{A}$. If $G$ is classical generator for $D^b_{\operatorname{coh}}(\mathcal{A})$, then $\mathbb{R}\pi_\ast G$ is a classical generator for $D^b_{\operatorname{coh}}(X)$.
\end{introthm}

There are interesting examples where Theorem~\ref{introthm:descent_for_coherent_algebras} can be applied. This includes categorical resolutions\footnote{See \cite{Kuznetsov/Lunts:2015, Lunts:2010, Kuznetsov:2007b}.}, homological projective duality\footnote{See \cite{Kuznetsov:2006} and more recently \cite{Perry:2019,Kuznetsov/Perry:2021a}.}, and noncommutative crepant resolutions\footnote{See \cite{VdB:2004, Stafford/VdB:2008,Iyama/Wemyss:2013, Dao/Faber/Ingalls:2015}.}. These are special cases of noncommutative schemes which could be leveraged for studying generation in the commutative setting. 

The next result gives a way to make this precise in mild situations.

\begin{introcor}\label{introcor:rouquier_bound_nc_descent}(see Corollary~\ref{cor:rouquier_bound_nc_descent})
    Let $R$ be a commutative Noetherian $J\textrm{-}2$ ring of finite Krull dimension and $\pi \colon R \to S$ a finite ring morphism such that $\pi_\ast S$ has full support as an $R$-module. If $G$ is a classical generator for $D^b_{\operatorname{coh}}(S)$ which is a bounded complex of $(S,S)$-bimodules, then
    \begin{displaymath}
        \dim D^b_{\operatorname{coh}}(R)\leq \operatorname{level}^{\pi_\ast G} (R) \cdot \big(\dim D^b_{\operatorname{coh}}(S) + 1\big) - 1.
    \end{displaymath}
\end{introcor}

Another interesting problem is to provide lower bounds for the Rouquier dimension of a triangulated category. Our last result provides lower bounds for the Rouquier dimension of the bounded derived category of a coherent $\mathcal{O}_X$-algebras on a general class of schemes. 

\begin{introthm}\label{introthm:}(see Theorem~\ref{thm:lower_bound_algebras})
    Let $X$ be a scheme of finite Krull dimension which is integral, Jacobson, catenary, $J\textrm{-}2$, and Noetherian. If $\mathcal{A}$ is a coherent $\mathcal{O}_X$-algebra with full support, then the Rouquier dimension of $D^b_{\operatorname{coh}}(\mathcal{A})$ is at least the Krull dimension of $X$.
\end{introthm}

Theorem~\ref{thm:lower_bound_algebras} is applicable to many cases of interest. This includes not only coherent algebras with full support over a variety, but also for such algebras over rings of mixed characteristic (i.e over $\mathbb{Z}$). It is worthwhile to note that in the commutative setting a similar bound holds for varieties over a field, see \cite[Proposition 7.16]{Rouquier:2008}. 

\begin{ack}
    The authors would like to thank Ryo Takahashi for clarifying Lemma~\ref{lem:noncommutative_projection_formula}, Timothy De Deyn for comments in an earlier draft, and the anonymous referee for many helpful suggestions. Souvik Dey was partially
    supported by the Charles University Research Center program No. UNCE/24/SCI/022
    and a grant GA CR 23-05148S from the Czech Science Foundation. 
\end{ack}

\section{Generation}
\label{sec:generation}

This section revisits two crucial notions of generation in triangulated categories, with a primary focus on those categories constructed from quasi-coherent sheaves on a Noetherian scheme. For foundational information on generation, please see \cite{BVdB:2003, Rouquier:2008, ABIM:2010}. We fix a triangulated category $\mathcal{T}$ with shift functor $[1]\colon \mathcal{T} \to \mathcal{T}$. Unless otherwise stated, all rings considered are Noetherian, and a `module' means a \textit{right module}.

\begin{notation}
    Let $X$ be a Noetherian scheme.
    \begin{itemize}
        \item $D_{\operatorname{Qcoh}}(X)$ is the derived category of complexes of $\mathcal{O}_X$-modules whose cohomology is quasi-coherent
        \item $D^b_{\operatorname{coh}}(X)$ is the derived category of bounded complexes of $\mathcal{O}_X$-modules whose cohomology is coherent.
    \end{itemize}
    By abuse of notation, we let $D^b_{\operatorname{coh}}(R)$ denote $D^b_{\operatorname{coh}}(\operatorname{Spec}(R))$ if $R$ is a commutative Noetherian ring, and similarly for other associated categories.
\end{notation}

\begin{definition}
    A full triangulated subcategory $\mathcal{S}$ of $\mathcal{T}$ is
    \textbf{thick} whenever it is closed under retracts. The smallest thick subcategory in $\mathcal{T}$ containing
    $\mathcal{S}$ is denoted $\langle \mathcal{S} \rangle$. If $\mathcal{S}$ consist of a single object $S$, then we set $\langle \mathcal{S} \rangle=\langle S \rangle$.
\end{definition}

\begin{definition}
    Let $\mathcal{S}$ be a subcategory of $\mathcal{T}$.
    \begin{enumerate}
    \item $\langle \mathcal{S} \rangle_0 = 0$
    \item $\langle \mathcal{S} \rangle_1$ is the full subcategory containing $\mathcal{S}$ closed under shifts and retracts of finite coproducts
    \item For $n\geq 2$, $\langle \mathcal{S} \rangle_n$ denotes the full subcategory of objects which are retracts of an object $E$ appearing in a distinguished triangle
    \begin{displaymath}
        A \to E \to B \to A[1]
    \end{displaymath}
    where $A\in \langle \mathcal{S} \rangle_{n-1}$ and $B\in
    \langle \mathcal{S} \rangle_1$.
\end{enumerate}
    If $\mathcal{S}$ consist of a single object $G$, then we set $\langle \mathcal{S} \rangle_n =\langle G \rangle_n$.
\end{definition}

\begin{remark}
    It can be seen that $\bigcup_{n=0}^\infty \langle \mathcal{S} \rangle_n$ is contained in $\langle \mathcal{S} \rangle$ because of how each $\langle \mathcal{S} \rangle_n$ is defined. An induction argument will show that $\bigcup_{n=0}^\infty \langle \mathcal{S} \rangle_n$ is a thick subcategory of $\mathcal{T}$ containing $\mathcal{S}$, and so, $\langle \mathcal{S} \rangle$ is a subcategory of $\bigcup_{n=0}^\infty \langle \mathcal{S} \rangle_n$. Tying this together gives us an exhaustive filtration:
    \begin{displaymath}
        \langle \mathcal{S} \rangle_0 \subseteq \langle \mathcal{S} \rangle_1 \subseteq \cdots \subseteq \bigcup_{n=0}^\infty \langle \mathcal{S} \rangle_n = \langle \mathcal{S} \rangle.
    \end{displaymath}
\end{remark}

\begin{definition}
    Let $E,G$ be objects of $\mathcal{T}$. The object $G$ \textbf{finitely
    builds} $E$ if $E$ belongs to $\langle G \rangle$, and if every object of
    $\mathcal{T}$ is finitely built by $G$, then we say $G$ is a
    \textbf{classical generator}. Moreover, if there exists an $n\geq 0$ such
    that $\langle G \rangle_n = \mathcal{T}$, then $G$ is called a
    \textbf{strong generator}.
\end{definition}

\begin{remark}\label{rmk:j_conditions}
    Let $X$ be a Noetherian scheme. There is an interesting connection with openness of the regular locus\footnote{The collection of points $p$ in $X$ such that $\mathcal{O}_{X,p}$ is regular.} of $X$ and existence of classical generators for $D^b_{\operatorname{coh}}(X)$. Specifically, it has been shown that every closed integral subscheme $Z$ of $X$ has an open regular locus if, and only if, $D^b_{\operatorname{coh}}(Z)$ admits a classical generator (see \cite[Theorem 1.1]{Dey/Lank:2024b}). This particular result was initially studied in the affine setting \cite{Iyengar/Takahashi:2019}. The openness of a regular locus for a scheme has be of interest outside of generation problems, and leads to the following well-studied notions (see \cite[\href{https://stacks.math.columbia.edu/tag/07P6}{Tag 07P6}]{StacksProject} and \cite[\href{https://stacks.math.columbia.edu/tag/07R2}{Tag 07R2}]{StacksProject} for details):
    \begin{enumerate}
        \item $X$ is \textbf{$\textrm{J-0}$} if the regular locus of $\mathcal{X}$ contains a nonempty open subset of $X$
        \item $X$ is \textbf{$\textrm{J-1}$} if the regular locus of $X$ is open
        \item $X$ is \textbf{$\textrm{J-2}$} if the regular locus of $Y$ is open for every morphism $Y\to X$ that is locally of finite type.
    \end{enumerate}
\end{remark}

\begin{example}\label{ex:j2_scheme_classical_generator}
    \begin{enumerate}
        \item \cite[Theorem 4.15]{Elagin/Lunts/Schnurer:2020} If $X$ is a Noetherian $J\textrm{-}2$, then $D^b_{\operatorname{coh}}(X)$ of  admits a classical generator.
        \item \cite[Main Theorem]{Aoki:2021} If $X$ is a Noetherian
        quasi-excellent separated scheme of finite Krull dimension, then
        $D^b_{\operatorname{coh}}(X)$ admits a strong generator.
        \item \cite[Corollary 3.9]{BILMP:2023} If $X$ is a Noetherian $F$-finite scheme, $G$ is a compact generator, and $e \gg 0$, then $F_\ast^e G$ is a classical generator for $D^b_{\operatorname{coh}}(X)$.
    \end{enumerate}
\end{example}

\begin{definition}
    Suppose $A,B,C$ are objects of $\mathcal{T}$. If $A$ belongs to $\langle B \rangle$, then we say the \textbf{level} of $A$ with respect to $B$ is the minimal $n$ required such that $A$ belongs to $\langle B \rangle_n$. This value is denoted by $\operatorname{level}_\mathcal{T}^B(A)$. If $C$ is a strong generator for $\mathcal{T}$, then its \textbf{generation time} is the minimal $m$ needed so that the level of an object $A$ in $\mathcal{T}$ with respect to $C$ is at most $m+1$. The smallest integer $d$ such that there exists a strong generator $G$ in $\mathcal{T}$ whose generation time is $d+1$ is called the \textbf{Rouquier dimension} of $\mathcal{T}$, and it is denoted $\dim \mathcal{T}$.
\end{definition}

\section{Noncommutative schemes}
\label{sec:noncommutative_schemes_background}

This section draws directly on content found in \cite{Yekutieli/Zhang:2006, Elagin/Lunts/Schnurer:2020, Burban/Drozd/Gavran:2017a}. For further details, the reader is encouraged to refer to these sources; especially \cite[Section 4]{Elagin/Lunts/Schnurer:2020} where background is sourced from. Consider a Noetherian scheme $X$. Let $\mathcal{A}$ be an $\mathcal{O}_X$-algebra. The notions of coherent and quasi-coherent $\mathcal{A}$-modules are defined in the obvious way. We denote the respective full subcategories of quasi-coherent and coherent sheaves in the category $\operatorname{Mod}\mathcal{A}$ of $\mathcal{A}$-modules by $\operatorname{Qcoh}\mathcal{A}$ and $\operatorname{coh}\mathcal{A}$. Note that $\operatorname{coh}\mathcal{A}$ is a full abelian subcategory of $\operatorname{Mod}\mathcal{A}$, which is also abelian. Since $\mathcal{A}$ is an $\mathcal{O}_X$-algebra, there exists a corresponding canonical map $\pi \colon \mathcal{O}_X \to \mathcal{A}$. If $E \in \operatorname{Mod}\mathcal{A}$, we can view $E$ as a sheaf of $\mathcal{O}_X$-modules by locally restricting scalars via $\pi$. We denote the resulting sheaf of $\mathcal{O}_X$-modules as $\pi_\ast E$.

\begin{remark}\label{rmk:ELS_lemma_quasi_coherent_modules}
    The following is \cite[Lemma 4.1]{Elagin/Lunts/Schnurer:2020}.
    \begin{enumerate}
        \item Assume $\mathcal{A}\in \operatorname{Qcoh}X$. An $\mathcal{A}$-module $E$ is quasi-coherent if, and only if, $\pi_\ast E$ is quasi-coherent as an $\mathcal{O}_X$-module.
        \item Assume $\mathcal{A}\in \operatorname{coh}X$. An $\mathcal{A}$-module $E$ is coherent if, and only if, $\pi_\ast E$ is coherent as an $\mathcal{O}_X$-module.
    \end{enumerate}
\end{remark}

\begin{definition}
    \begin{enumerate}
        \item A \textbf{noncommutative scheme} is a pair $(Y,\mathcal{B})$ where $Y$ is a scheme and $\mathcal{B}$ is a quasi-coherent sheaf of $\mathcal{O}_Y$-algebras. 
        \item A \textbf{morphism} of noncommutative schemes $f\colon (Y,\mathcal{B}) \to (X,\mathcal{A})$ is a pair $(f_X, f^\#)$ where $f_X \colon Y \to X$ is a morphism of schemes and$f^\#$ is a morphism of $f^{-1}_X \mathcal{A}$-algebras $f^{-1}_X \mathcal{A} \to \mathcal{B}$. 
    \end{enumerate}
\end{definition}

\begin{remark}\label{rmk:canonical_morphism}
    \begin{enumerate}
        \item If $X$ is a scheme, then the pair $(X,\mathcal{O}_X)$ is a noncommutative scheme. Assume $\mathcal{A}$ is a quasi-coherent $\mathcal{O}_X$-algebra. Consider the morphism of noncommutative schemes: 
        \begin{displaymath}
            (1_X, \pi)\colon (X,\mathcal{A})\to (X,\mathcal{O}_X).
        \end{displaymath}
        This pair is referred to as the \textbf{structure morphism}, and we abuse notation to denote it as $\pi$. \item Let $i\colon Z \to X$ be a closed immersion $i\colon Z \to X$. This induces a morphism of noncommutative schemes $(Z,i^\ast \mathcal{A}) \to (X,\mathcal{A})$, which we abuse notation and denote by $i$ as well. There exists a commutative diagram of noncommutative schemes:
        \begin{displaymath}
            \begin{tikzcd}[ampersand replacement=\&]
                {(Z,i^\ast \mathcal{A})} \& {(X,\mathcal{A})} \\
                Z \& X.
                \arrow["i", from=1-1, to=1-2]
                \arrow["\theta"', from=1-1, to=2-1]
                \arrow["\pi", from=1-2, to=2-2]
                \arrow["i"', from=2-1, to=2-2]
            \end{tikzcd}
        \end{displaymath}
        Here, $\theta$ is the structure morphism $\mathcal{O}_Z \to j^\ast \mathcal{A}$.
    \end{enumerate}
\end{remark}

\begin{remark}
    Let $f \colon Y \to X$ be a morphism of schemes, and set $\mathcal{B}:= f^\ast \mathcal{A}$
    \begin{enumerate}
        \item Remark~\ref{rmk:ELS_lemma_quasi_coherent_modules} tells us $\mathcal{B}$ is an $\mathcal{O}_Y$-algebra. Moreover, if $\mathcal{A}$ is (quasi-)coherent, then so is $\mathcal{B}$.
        \item There are the usual adjunctions $f^\ast \colon \operatorname{Mod}\mathcal{A} \to \operatorname{Mod}\mathcal{B}$ and $f_\ast \colon \operatorname{Mod}\mathcal{B} \to \operatorname{Mod}\mathcal{A}$ which are defined in the usual way. Note that $f^\ast$ is left adjoint to $f_\ast$. These functors commutes with the forgetful functor to modules over the underlying schemes.
        \item If $\mathcal{A}$ is quasi-coherent, then the adjunction above restrict to one between quasi-coherent sheaves.
        \item If $\mathcal{A}$ is coherent, then $f^\ast$ restricts to $f^\ast \colon \operatorname{coh}\mathcal{A}\to \operatorname{coh}\mathcal{B}$.
    \end{enumerate}
    For further details, please refer to \cite[Section 4]{Elagin/Lunts/Schnurer:2020}.
\end{remark}

The following Verdier localization sequence will be vital for the proof of Theorem~\ref{thm:descent_for_coherent_algebras}. Assume $\mathcal{A}$ is a coherent $\mathcal{O}_X$-algebra. Denote the derived category of complexes of coherent $\mathcal{A}$-modules by $D_{\operatorname{coh}}(\mathcal{A})$, and $D^b_{\operatorname{coh}}(\mathcal{A})$ for the full subcategory of bounded complexes. If $Z$ is a closed subscheme of $X$, we denote by $D^b_{\operatorname{coh},Z} (\mathcal{A})$ the full subcategory of $D^b_{\operatorname{coh}}(\mathcal{A})$ whose cohomology sheaves are supported in $Z$ when viewed as $\mathcal{O}_X$-modules. Note that $D^b_{\operatorname{coh},Z}(\mathcal{A})$ is a thick subcategory of $D^b_{\operatorname{coh}}(\mathcal{A})$.

\begin{remark}\label{rmk:nc_verdier_localization}
    If $j\colon U \to X$ is an open immersion, then there exists a Verdier localization sequence
    \begin{displaymath}
        D^b_{\operatorname{coh},Z} (\mathcal{A}) \to D^b_{\operatorname{coh}}(\mathcal{A}) \xrightarrow{j^\ast} D^b_{\operatorname{coh}}(j^\ast \mathcal{A})
    \end{displaymath}
    where $Z:= X \setminus Y$. Note that the exact functor $j^\ast\colon D^b_{\operatorname{coh}}(\mathcal{A}) \to D^b_{\operatorname{coh}}(j^\ast \mathcal{A})$ is a small abuse of notation, but it is defined by applying $j^\ast$ to each component of a complex in $D^b_{\operatorname{coh}}(\mathcal{A})$. This is \cite[Theorem 4.4]{Elagin/Lunts/Schnurer:2020}.
\end{remark}

\begin{example}\label{ex:ELS_j2_classical_generator}
    \begin{enumerate}
        \item Let $X$ be a Noetherian $J\textrm{-}2$ scheme. If $\mathcal{A}$ is a coherent $\mathcal{O}_X$-algebra, then $D^b_{\operatorname{coh}}(\mathcal{A})$ admits a classical generator. This is \cite[Theorem 4.15]{Elagin/Lunts/Schnurer:2020}.
        \item Let $X$ be a separated scheme of finite type over a perfect field. Then $D^b_{\operatorname{coh}}(\mathcal{A})$ admits a strong generator for any coherent $\mathcal{O}_X$-algebra, see \cite[Example 3.9]{DeDeyn/Lank/ManaliRahul:2024a}. The affine setting was investigated \cite[Remark 2.6]{Elagin/Lunts/Schnurer:2020}.
    \end{enumerate}
\end{example}

\begin{remark}\label{rmk:nc_closed_immersion}
    Suppose $i \colon Z \to X$ is a closed immersion. Note $i_\ast \colon \operatorname{Mod}Z \to \operatorname{Mod}X$ is exact and preserves coherence. If $\mathcal{A}$ is a coherent $\mathcal{O}_X$-module, then $i_\ast \colon \operatorname{coh}i^\ast \mathcal{A} \to \operatorname{coh}\mathcal{A}$ is well-defined and exact. Hence, this gives an induced (exact) derived functor $i_\ast \colon D^b_{\operatorname{coh}}(i^\ast \mathcal{A})\to D^b_{\operatorname{coh}}(\mathcal{A})$. Moreover, if $G$ is a classical generator for $D^b_{\operatorname{coh}}(i^\ast \mathcal{A})$, then $i_\ast G$ is a classical generator for $D^b_{\operatorname{coh},Z} (\mathcal{A})$. This is \cite[Proposition 4.6]{Elagin/Lunts/Schnurer:2020}.
\end{remark}

\section{Descent along structure morphism}
\label{sec:descent_finite_algebras}

This section investigates the preservation of classical generators through the derived pushforward along the structure morphism from a noncommutative scheme to its underlying scheme. We start with a few elementary lemmas.

\begin{lemma}\label{lem:coherent_sheaves_exact_on_module_categories}
    Let $X$ be a Noetherian scheme and $\mathcal{A}$ be a coherent $\mathcal{O}_X$-algebra. Then the canonical map $\pi \colon \mathcal{O}_X \to \mathcal{A}$ induces a exact functor $\pi_\ast \colon \operatorname{coh}\mathcal{A} \to \operatorname{coh}X$. 
\end{lemma}

\begin{proof}
    If $E\in \operatorname{coh}\mathcal{A}$, then $\pi_\ast E\in \operatorname{coh}X$ by Remark~\ref{rmk:ELS_lemma_quasi_coherent_modules}. This gives us a functor $\pi_\ast \colon \operatorname{coh}\mathcal{A} \to \operatorname{coh}X$. To verify exactness, consider a short exact sequence in $\operatorname{coh}\mathcal{A}$:
    \begin{displaymath}
        0 \to A \to B \to C \to 0.
    \end{displaymath}
    This is a local problem on $X$, so we may assume $X=\operatorname{Spec}(R)$ for some commutative Noetherian ring $R$, and $\mathcal{A}$ can be replaced by a finite $R$-algebra $S$. However, $\pi \colon \operatorname{mod}S \to \operatorname{mod}R$ is exact, as the restriction of scalars does not alter the underlying abelian groups. This completes the proof.
\end{proof}

\begin{lemma}\label{lem:cohernet_sheaves_exact_triangulated}
    Let $X$ be a Noetherian scheme and $\mathcal{A}$ be a coherent $\mathcal{O}_X$-algebra. Then the canonical map $\pi \colon \mathcal{O}_X \to \mathcal{A}$ induces a exact functor $\pi_\ast \colon D^b_{\operatorname{coh}}(\mathcal{A}) \to D^b_{\operatorname{coh}}(X)$.
\end{lemma}

\begin{proof}
    Note the canonical map $\pi \colon \mathcal{O}_X \to \mathcal{A}$ induces a exact functor $\pi_\ast \colon \operatorname{coh}\mathcal{A} \to \operatorname{coh}X$, see Lemma~\ref{lem:coherent_sheaves_exact_on_module_categories}. There exists an induced functor on derived categories $D_{\operatorname{coh}}(\mathcal{A})\to D_{\operatorname{coh}}(X)$, which we will also denote by $\pi_\ast$\footnote{$D_{\operatorname{coh}}(\mathcal{A})\to D_{\operatorname{coh}}(X)$ is defined by applying $\pi_\ast$ component wise on complex in $D_{\operatorname{coh}}(\mathcal{A})$.}. To see $\pi_\ast$ preserves bounded cohomology, this follows from the fact that $\pi$ is exact.
\end{proof}

\begin{remark}\label{rmk:classical_generation_from_irreducible_components}
    Let $X$ be a Noetherian scheme. If $X=\bigcup^n_{i=1} Z_i$ denotes the maximal irreducible components and $G_i \in D^b_{\operatorname{coh}}(Z_i)$ is a classical generator for each $1\leq i \leq n$, then $\oplus_{i=1}^n \mathbb{R} \pi_{i,\ast} G_i$ is a classical generator for $D^b_{\operatorname{coh}}(X)$ where $\pi_i \colon Z_i \to X$ is the closed immersion. This can be shown from \cite[Example 3.4]{Dey/Lank:2024}.
\end{remark}

\begin{theorem}\label{thm:descent_for_coherent_algebras}
    Let $X$ be a Noetherian $J\textrm{-}2$ scheme of finite Krull dimension. Suppose $\mathcal{A}$ is a coherent $\mathcal{O}_X$-algebra with full support and canonical map $\pi \colon \mathcal{O}_X \to \mathcal{A}$. If $G$ is classical generator for $D^b_{\operatorname{coh}}(\mathcal{A})$, then $\mathbb{R}\pi_\ast G$ is a classical generator for $D^b_{\operatorname{coh}}(X)$.
\end{theorem}

\begin{proof}
    Any closed subscheme $Z$ of $X$ is Noetherian $J\textrm{-}2$. We will prove the claim by Notherian induction on $X$. If $X=\emptyset$, there is nothing to check, so we can impose that $X$ is non-empty. Without loss of generality, we can assume the claim holds for any properly contained closed subscheme $Z$ of $X$. 

    Consider the case where $X$ is an integral scheme. Then
    $\operatorname{reg}(X)$ is a nonempty open subscheme of $X$ as it contains the generic point. Let $j\colon U
    \to X$ be an open immersion, with $U$ an affine open subscheme contained in
    $\operatorname{reg}(X)$. There is a a classical generator $E$ for
    $D^b_{\operatorname{coh}}(X)$ because $X$ is $J\textrm{-}2$, see \cite[Theorem 4.15]{Elagin/Lunts/Schnurer:2020}. Note that $j^\ast \pi_\ast G$ is a perfect complex on $U$ with full support because $U$ is regular and $\pi_\ast  G$ has full support on $X$. This tells us that $j^\ast \pi_\ast G$ is a classical
    generator for $D^b_{\operatorname{coh}}(U)$, see \cite[Lemma 1.2]{Neeman:1992}.
    Thus, $j^\ast \pi_\ast G$ finitely builds $j^\ast E$ in
    $D^b_{\operatorname{coh}}(U)$. Consequently, there exists $A\in
    D^b_{\operatorname{coh}}(Z)$, where $Z$ is a closed subscheme contained in
    $X\setminus U$, with a closed immersion $i\colon Z \to X$, such that $E$ is
    finitely built by $\pi_\ast G\oplus i_\ast A$, see \cite[Lemma 4.1.2]{Lank:2024}. 
    
    Next, we demonstrate that $\pi_\ast G$
    finitely builds $i_\ast A$ in $D^b_{\operatorname{coh}}(X)$. Let $\theta\colon \mathcal{O}_Z \to i^\ast \mathcal{A}$ be the structure morphism. There exists a commutative diagram of exact functors\footnote{This follows from Remark~\ref{rmk:canonical_morphism}, \cite[Proposition 3.12]{Burban/Drozd/Gavran:2017}, and Lemma~\ref{lem:cohernet_sheaves_exact_triangulated}.}:
    \begin{displaymath}
        \begin{tikzcd}
            {D^b_{\operatorname{coh}}(Z)} & {D^b_{\operatorname{coh}}(X)} \\
            {D^b_{\operatorname{coh}}(i^\ast \mathcal{A})} & {D^b_{\operatorname{coh}}(\mathcal{A})}.
            \arrow["{i_\ast}", from=1-1, to=1-2]
            \arrow["{i_\ast}", from=2-1, to=2-2]
            \arrow["{\pi_\ast}", from=2-2, to=1-2]
            \arrow["{\theta_\ast}", from=2-1, to=1-1]
        \end{tikzcd}
    \end{displaymath}
    Note that $D^b_{\operatorname{coh}}(i^\ast \mathcal{A})$ admits a classical because $Z$ is $J\textrm{-}2$ and $i^\ast \mathcal{A}$ is a coherent $\mathcal{O}_Z$-algebra, see \cite[Theorem 4.15]{Elagin/Lunts/Schnurer:2020}. Since $Z$ is a properly contained subscheme of $X$, the induction hypothesis ensures that if $G^\prime$ is a classical generator for $D^b_{\operatorname{coh}}(i^\ast \mathcal{A})$, then $\theta_\ast G^\prime$ is a classical generator for $D^b_{\operatorname{coh}}(Z)$. Consequently, $(i \circ \theta)_\ast G^\prime$ finitely builds $i_\ast A$ in $D^b_{\operatorname{coh}}(X)$. Moreover, since $G$ finitely builds $i_\ast  G^\prime$ in $D^b_{\operatorname{coh}}(\mathcal{A})$, we have that $\pi_\ast G$ finitely builds $(\pi \circ i)_\ast G^\prime$ in $D^b_{\operatorname{coh}}(X)$. Hence, $\pi_\ast G$ finitely builds $i_\ast A$ in $D^b_{\operatorname{coh}}(X)$ and, consequently, finitely builds $E$.

    Now we work in the general case where $X$ is not integral. It suffices to show that $i_\ast D^b_{\operatorname{coh}}(Z)$ is contained in $\langle\pi_\ast D^b_{\operatorname{coh}}(\mathcal{A}) \rangle$ for each closed immersion $i\colon Z \to X$ from an irreducible component, see Remark~\ref{rmk:classical_generation_from_irreducible_components}. Let $\theta\colon \mathcal{O}_Z \to i^\ast \mathcal{A}$ be the structure morphism. Our work above tells us that $\langle \theta_\ast D^b_{\operatorname{coh}}(i^\ast \mathcal{A}) \rangle = D^b_{\operatorname{coh}}(Z)$ because $Z$ is integral. However, we know that $i_\ast D^b_{\operatorname{coh}}(i^\ast \mathcal{A})$ is contained in $D^b_{\operatorname{coh}}(\mathcal{A})$ (see Remark~\ref{rmk:nc_closed_immersion}), and so, $i_\ast D^b_{\operatorname{coh}}(Z)$ belongs to $\langle \pi_\ast D^b_{\operatorname{coh}}(\mathcal{A}) \rangle$ because $i_\ast \circ \theta_\ast = \pi_\ast \circ i_\ast$. This completes the proof.
\end{proof}

\begin{corollary}
    Let $f \colon Y \to X$ be a proper surjective morphism of Noetherian $J\textrm{-}2$ schemes of finite Krull dimension. Suppose $\mathcal{A}$ is a coherent $\mathcal{O}_Y$-algebra with canonical map $\pi \colon \mathcal{O}_Y \to \mathcal{A}$. If $G$ is a classical generator for $D^b_{\operatorname{coh}}(\mathcal{A})$, then $\mathbb{R}(f \circ \pi)_\ast G$ is a classical generator for $D^b_{\operatorname{coh}}(X)$.
\end{corollary}

\begin{proof}
    This is Theorem~\ref{thm:descent_for_coherent_algebras} coupled with \cite[Corollary 3.12]{Dey/Lank:2024}.
\end{proof}

\begin{lemma}\label{lem:noncommutative_projection_formula}
    Let $R,S$ be Noetherian rings where $R$ is commutative and $\phi \colon R \to S$ be a ring homomorphism. If $E\in D_{\operatorname{Qcoh}}(R)$ and $A\in D_{\operatorname{Qcoh}}(S)$ is a complex of $(S,S)$-bimodules, then there exists a isomorphism in $D_{\operatorname{Qcoh}}(R)$:
    \begin{displaymath}
        E \overset{\mathbb{L}}{\otimes}_R \phi_\ast A \cong \phi_\ast (\mathbb{L}\phi^\ast E \overset{\mathbb{L}}{\otimes}_S A).
    \end{displaymath}
\end{lemma}

\begin{proof}
    We can realize $\phi_\ast\colon D_{\operatorname{Qcoh}}(S) \to D_{\operatorname{Qcoh}}(R)$ and $\mathbb{L}\phi^\ast \colon D_{\operatorname{Qcoh}}(R) \to D_{\operatorname{Qcoh}}(S)$ as the derived tensor product functors\footnote{These are respectively the derived functors for restriction and extension of scalars.}:
    \begin{displaymath}
        \begin{aligned}
            \phi_\ast(-) &= (-) \overset{\mathbb{L}}{\otimes_S} (S_R),
            \\ \mathbb{L} \phi^\ast(-) &=(-)\overset{\mathbb{L}}{\otimes}_R S,
        \end{aligned}
    \end{displaymath}
    where $S_R$ denotes $S$ as a right $R$-module via $\phi$. There exists an isomorphism in $D_{\operatorname{Qcoh}}(R)$:
    \begin{displaymath}
        \begin{aligned}
            \phi_\ast ( \mathbb{L} \phi^\ast E \overset{\mathbb{L}}{\otimes}_S A) &\cong((E\overset{\mathbb{L}}{\otimes}_R S) \overset{\mathbb{L}}{\otimes}_S A)\overset{\mathbb{L}}{\otimes}_S (S_R) 
            \\&\cong E\overset{\mathbb{L}}{\otimes}_R ((S \overset{\mathbb{L}}{\otimes}_S A)\overset{\mathbb{L}}{\otimes}_S (S_R))
            \\&\cong E\overset{\mathbb{L}}{\otimes}_R (A\overset{\mathbb{L}}{\otimes}_S (S_R))
            \\&\cong E \overset{\mathbb{L}}{\otimes}_R \pi_\ast A
        \end{aligned}
    \end{displaymath}
\end{proof}

\begin{corollary}\label{cor:rouquier_bound_nc_descent}
    Let $R$ be a commutative Noetherian $J\textrm{-}2$ ring of finite Krull dimension and $\pi \colon R \to S$ a finite ring morphism such that $\pi_\ast S$ has full support as an $R$-module. If $G$ is a classical generator for $D^b_{\operatorname{coh}}(S)$ which is a bounded complex of $(S,S)$-bimodules, then
    \begin{displaymath}
        \dim D^b_{\operatorname{coh}}(R)\leq \operatorname{level}^{\pi_\ast G} (R) \cdot \big(\dim D^b_{\operatorname{coh}}(S) + 1\big) - 1.
    \end{displaymath}
\end{corollary}

\begin{proof}
    If $D^b_{\operatorname{coh}}(S)$ has infinite Rouquier dimension, then there is nothing to check, so without loss of generality we may assume this value is finite. Suppose $G$ is a strong generator for $D^b_{\operatorname{coh}}(S)$ with generation time is $g$. By Theorem~\ref{thm:descent_for_coherent_algebras}, $\pi_\ast G$ is a classical generator for $D^b_{\operatorname{coh}}(R)$. Choose $n\geq 0$ such that $R\in \langle \pi_\ast G\rangle_n$. Let $P$ be in $\operatorname{perf}(R)$. Our hypothesis that $G$ is a complex of $(S,S)$-bimodules ensures that $\pi_\ast (\mathbb{L}\pi^\ast P \overset{\mathbb{L}}{\otimes}_S G)$ is isomorphic to $\pi_\ast G \overset{\mathbb{L}}{\otimes}_R P= P \overset{\mathbb{L}}{\otimes}_R S$, see Lemma~\ref{lem:noncommutative_projection_formula}. If we tensor with $P$ and use this isomorphism, then we see that $P\in \langle \pi_\ast (G\overset{\mathbb{L}}{\otimes}_S \mathbb{L}\pi^\ast P) \rangle_n$. In particular, this tells us that $\operatorname{perf}R\subseteq \langle \pi_\ast G \rangle_{n(g+1)}$. The desired claim follows from \cite[Theorem 1.1]{Lank/Olander:2024} (as well as \cite[Theorem 7]{Opperman/Stovicek:2012}).
\end{proof}

\begin{theorem}\label{thm:lower_bound_algebras}
    Let $X$ be a scheme of finite Krull dimension which is integral, Jacobson\footnote{See \cite[\href{https://stacks.math.columbia.edu/tag/01P1}{Tag 01P1}]{StacksProject}.}, catenary\footnote{See \cite[\href{https://stacks.math.columbia.edu/tag/02IV}{Tag 02IV}]{StacksProject}.}, $J\textrm{-}2$, and Noetherian. If $\mathcal{A}$ is a coherent $\mathcal{O}_X$-algebra with full support, then the Rouquier dimension of $D^b_{\operatorname{coh}}(\mathcal{A})$ is at least the Krull dimension of $X$.
\end{theorem}

\begin{proof}
    We know that $D^b_{\operatorname{coh}}(\mathcal{A})$ admits a classical generator, see \cite[Theorem 4.15]{Elagin/Lunts/Schnurer:2020}. If $D^b_{\operatorname{coh}}(\mathcal{A})$ has no strong generators, then there is nothing to check, so we can assume $D^b_{\operatorname{coh}}(\mathcal{A})$ admits a strong generator. Suppose $G$ is a strong generator for $D^b_{\operatorname{coh}}(\mathcal{A})$ with minimal generation time. The regular locus of $X$ is nonempty and open, so we can find an affine open immersion $i \colon U \to X$ where $U$ is nonempty. The locus of points $p$ where both $(\pi_\ast G)_p$ and $\mathcal{O}_{X,p}$ finitely build one another in at most one cone in $D^b_{\operatorname{coh}}(\mathcal{O}_{X,p})$ is open in $U$, see \cite[Lemma 3.1]{Dey/Lank:2024} or \cite[Proposition 3.5]{Letz:2021}. 

    Let $V$ be an affine open contained in the intersection of $U$ and such points.\footnote{This is nonempty because both objects do not vanish at the generic point.} Denote by $j\colon V \to X$ the open immersion. It follows that both $j^\ast \pi_\ast G$ and $\mathcal{O}_V$ finitely build one another in one cone in $D^b_{\operatorname{coh}}(V)$, see \cite[Corollary 3.4]{Letz:2021}. Note $V$ is contained in the regular locus of $X$. This tells us $V$ is a affine regular scheme of finite Krull dimension, so $\mathcal{O}_V$ is a strong generator whose generation time is equal to $\dim V$, see \cite[Corollary 4.3.13]{Letz:2020} or \cite[$\S 8.2$]{Christensen:1998}. Hence, $j^\ast \pi_\ast G$ has generation time coinciding with $\dim V$ in $D^b_{\operatorname{coh}}(V)$. 
    
    The hypothesis on $X$ tells us the Krull dimension of any nonempty affine open subscheme of $X$ coincides with that of $X$, see \cite[\href{https://stacks.math.columbia.edu/tag/0DRT}{Tag 0DRT}]{StacksProject}. This ensures that $j^\ast \pi_\ast G$ is a strong generator for $D^b_{\operatorname{coh}}(V)$ with generation time being $\dim X$. Choose a closed point $p$ in $V$ such that $\dim \mathcal{O}_{V,p}=\dim X$. As $V$ is regular, it follows that $\operatorname{level}^{\mathcal{O}_{V,p}}(\kappa(p)) = \dim X + 1$ see \cite[Corollary 4.3.13]{Letz:2020} coupled with \cite[\href{https://stacks.math.columbia.edu/tag/00OB}{Tag 00OB}]{StacksProject}. Let $s\colon \operatorname{Spec}(\kappa(p)) \to V$ be the closed immersion. There is a commutative diagram of exact functors:
    \begin{displaymath}
        \begin{tikzcd}
            {D^b_{\operatorname{coh}}(\kappa(p))} & {D^b_{\operatorname{coh}}(X)} \\
            {D^b_{\operatorname{coh}}(s^\ast j^\ast \mathcal{A})} & {D^b_{\operatorname{coh}}(j^\ast \mathcal{A})}.
            \arrow["{s_\ast}", from=1-1, to=1-2]
            \arrow["{s_\ast}", from=2-1, to=2-2]
            \arrow["{\pi_\ast^\prime}", from=2-2, to=1-2]
            \arrow["{\theta_\ast}", from=2-1, to=1-1]
        \end{tikzcd}
    \end{displaymath}
    where $\theta\colon \mathcal{O}_{\operatorname{Spec}}(\kappa(p)) \to s^\ast j^\ast \mathcal{A}$ and $\pi^\prime \colon \mathcal{O}_V \to j^\ast \mathcal{A}$ are the structure morphisms. 
    
    We see that $\theta_\ast s^\ast j^\ast \mathcal{A}$ is isomorphic to a complex of the form $\bigoplus_{n\in \mathbb{Z}} \mathcal{O}_{\operatorname{Spec}}(\kappa(p))^{\oplus r_n} [n]$. Recall that $j^\ast \colon D^b_{\operatorname{coh}}(\mathcal{A}) \to D^b_{\operatorname{coh}}(j^\ast \mathcal{A})$ is a Verdier localization, see Remark~\ref{rmk:nc_verdier_localization}. Suppose $Q$ is an object of $D^b_{\operatorname{coh}}(\mathcal{A})$ such that $j^\ast Q$ is isomorphic to $s_\ast s^\ast j^\ast \mathcal{A}$ in $D^b_{\operatorname{coh}}(j^\ast \mathcal{A})$. The following diagram commutes up to a natural isomorphism, see proof of \cite[Proposition 3.1]{DeDeyn/Lank/ManaliRahul:2024a}:
    \begin{displaymath}
        \begin{tikzcd}
            {D^b_{\operatorname{coh}}(j^\ast \mathcal{A})} & {D^b_{\operatorname{coh}}(U)} \\
            {D^b_{\operatorname{coh}}(\mathcal{A})} & {D^b_{\operatorname{coh}}(X)}.
            \arrow["{\pi^\prime_\ast}", from=1-1, to=1-2]
            \arrow["{\pi_\ast}", from=2-1, to=2-2]
            \arrow["{j^\ast}", from=2-2, to=1-2]
            \arrow["{j^\ast}", from=2-1, to=1-1]
        \end{tikzcd}
    \end{displaymath}
    This tells us that $\pi^\prime_\ast j^\ast G$ is isomorphic to $j^\ast \pi_\ast G$ in $D^b_{\operatorname{coh}}(V)$. Tying together these observations gives us the following string of inequalities: 
    \begin{displaymath}
        \begin{aligned}
            \dim D^b_{\operatorname{coh}}(\mathcal{A}) + 1 &= \operatorname{gen.time}(G) + 1 
            \\&\geq \operatorname{level}^G (Q) 
            \\&\geq \operatorname{level}^{j^\ast G} (j^\ast Q) 
            \\&\geq \operatorname{level}^{\pi^\prime_\ast j^\ast G} (\pi^\prime_\ast j^\ast Q) 
            \\&\geq \operatorname{level}^{j^\ast \pi_\ast G} (\pi^\prime_\ast j^\ast Q)
            \\&\geq \operatorname{level}^{\mathcal{O}_V} (\mathcal{O}_{\operatorname{Spec}}(\kappa(p))) \\&\geq \operatorname{level}^{\mathcal{O}_{V,p}}(\kappa(p)) =\dim X + 1.
        \end{aligned}
    \end{displaymath} 
    This completes the poof.
\end{proof}

\bibliographystyle{alpha}
\bibliography{mainbib}

\end{document}